\documentclass[12pt]{amsart}
\RequirePackage{mathtools}
\RequirePackage{amsopn}
\RequirePackage[normalem]{ulem} 
\RequirePackage{amsfonts}
\RequirePackage{paralist}
\RequirePackage{amssymb}
\RequirePackage{amsthm}
\RequirePackage{mathrsfs}
\usepackage[alphabetic]{amsrefs}
\renewcommand\MR[1]{\relax} 
\usepackage{tikz}
\usetikzlibrary{cd,decorations.pathmorphing,decorations.markings}
\mathtoolsset{showonlyrefs}
\newtheorem{thm}{Theorem}[section] 
\numberwithin{equation}{section}

\newtheorem{lem}[thm]{Lemma}
\newtheorem{prop}[thm]{Proposition}

\theoremstyle{definition}

\theoremstyle{remark}

\newtheorem{mycomment}[thm]{Comment}
{\end{mycomment}\endgroup}
\DeclareMathSymbol{\rtimes}{\mathbin}{AMSb}{"6F}

\newcommand\K{\mathcal{K}}

\makeatletter
\def\labelenumi{\textnormal{(\@alph\c@enumi)}}
\def\theenumi{\@alph \c@enumi}
\def\labelenumii{\textnormal{(\@roman\c@enumii)}}
\def\theenumii{\@roman \c@enumii}
\newcount\charno
\def\alphapart#1{\charno=96
\advance\charno by#1\char\charno}

\makeatother

\def\<{\langle}
\def\>{\rangle}
\let\ipscriptstyle=\scriptscriptstyle
\def\lipsqueeze{{\mskip -3.0mu}}
\def\ripsqueeze{{\mskip -3.0mu}}
\def\ipcomma{\nobreak\mathrel{,}\nobreak}
\newbox\ipstrutbox
\setbox\ipstrutbox=\hbox{\vrule height8.5pt
width 0pt}
\def\ipstrut{\copy\ipstrutbox}
\def\lip#1<#2,#3>{\mathopen{\relax_{\ipstrut\ipscriptstyle{
#1}}\lipsqueeze
\langle} #2\ipcomma #3 \rangle}
\def\blip#1<#2,#3>{\mathopen{\relax_{\ipstrut
\ipscriptstyle{ #1}}\lipsqueeze\bigl\langle} #2\ipcomma #3 \bigr\rangle}
\def\rip#1<#2,#3>{\langle #2\ipcomma #3
\rangle_{\ripsqueeze\ipstrut\ipscriptstyle{#1}}}
\def\brip#1<#2,#3>{\bigl\langle #2\ipcomma #3
\bigr\rangle_{\ripsqueeze\ipstrut\ipscriptstyle{#1}}}
\def\angsqueeze{\mskip -6mu}
\def\smangsqueeze{\mskip -3.7mu}
\def\trip#1<#2,#3>{\langle\smangsqueeze\langle #2\ipcomma #3
\rangle\smangsqueeze\rangle_{\ripsqueeze\ipstrut\ipscriptstyle{#1}}}
\def\btrip#1<#2,#3>{\bigl\langle\angsqueeze\bigl\langle #2\ipcomma
#3
\bigr\rangle
\angsqueeze\bigr\rangle_{\ripsqueeze\ipstrut\ipscriptstyle{#1}}}
\def\tlip#1<#2,#3>{\mathopen{\relax_{\ipstrut\ipscriptstyle{
#1}}\lipsqueeze \langle\smangsqueeze\langle} #2\ipcomma #3
\rangle\smangsqueeze\rangle}
\def\btlip#1<#2,#3>{\mathopen{\relax_{\ipstrut\ipscriptstyle{
#1}}\lipsqueeze
\bigl\langle\angsqueeze\bigl\langle} #2\ipcomma #3
\bigr\rangle\angsqueeze\bigr\rangle}

\def\ip(#1|#2){(#1\mid #2)}
\def\bip(#1|#2){\bigl(#1 \mid #2\bigr)}
\def\Bip(#1|#2){\Bigl( #1 \bigm| #2 \Bigr)}
\expandafter\ifx\csname BibSpec\endcsname\relax\else
\BibSpec{collection.article}{%
    +{}  {\PrintAuthors}                {author}
    +{,} { \textit}                     {title}
    +{.} { }                            {part}
    +{:} { \textit}                     {subtitle}
    +{,} { \PrintContributions}         {contribution}
    +{,} { \PrintConference}            {conference}
    +{}  {\PrintBook}                   {book}
    +{,} { }                            {booktitle}
    +{,} { }                            {series}
    +{,} { \voltext}                    {volume}
    +{,} { }                            {publisher}
    +{,} { }                            {organization}
    +{,} { }                            {address}
    +{,} { \PrintDateB}                 {date}
    +{,} { pp.~}                        {pages}
    +{,} { }                            {status}
    +{,} { \PrintDOI}                   {doi}
    +{,} { available at \eprint}        {eprint}
    +{}  { \parenthesize}               {language}
    +{}  { \PrintTranslation}           {translation}
    +{;} { \PrintReprint}               {reprint}
    +{.} { }                            {note}
    +{.} {}                             {transition}
}
\BibSpec{article}{%
    +{}  {\PrintAuthors}                {author}
    +{,} { \textit}                     {title}
    +{.} { }                            {part}
    +{:} { \textit}                     {subtitle}
    +{,} { \PrintContributions}         {contribution}
    +{.} { \PrintPartials}              {partial}
    +{,} { }                            {journal}
    +{}  { \textbf}                     {volume}
    +{}  { \PrintDatePV}                {date}
    +{,} { \eprintpages}                {pages}
    +{,} { }                            {status}
    +{,} { \PrintDOI}                   {doi}
    +{,} { available at \eprint}        {eprint}
    +{}  { \parenthesize}               {language}
    +{}  { \PrintTranslation}           {translation}
    +{;} { \PrintReprint}               {reprint}
    +{.} { }                            {note}
    +{.} {}                             {transition}
}
\BibSpec{book}{%
    +{}  {\PrintPrimary}                {transition}
    +{,} { \textit}                     {title}
    +{.} { }                            {part}
    +{:} { \textit}                     {subtitle}
    +{,} { \PrintEdition}               {edition}
    +{}  { \PrintEditorsB}              {editor}
    +{,} { \PrintTranslatorsC}          {translator}
    +{,} { \PrintContributions}         {contribution}
    +{,} { }                            {series}
    +{,} { \voltext}                    {volume}
    +{,} { }                            {publisher}
    +{,} { }                            {organization}
    +{,} { }                            {address}
    +{,} { pp.~}                        {pages}
    +{,} { \PrintDateB}                 {date}
    +{,} { }                            {status}
    +{}  { \parenthesize}               {language}
    +{}  { \PrintTranslation}           {translation}
    +{;} { \PrintReprint}               {reprint}
    +{.} { }                            {note}
    +{.} {}                             {transition}
}
\fi
\evensidemargin=\oddsidemargin
\newcommand{\axg}{A\rtimes_\alpha G\rtimes_{\what\alpha} G}

\newcommand{\adg}{A\rtimes_\delta G\rtimes_{\what\delta} G}

\newcommand\I{\mathscr{I}}

\newcommand{\ad}{\operatorname{Ad}}

\newcommand{\id}{\operatorname{id}}
\newcommand{\ibm}{imprimitivity bimodule}

\renewcommand{\bar}{\overline}
\newcommand{\spn}{\operatorname{span}}
\newcommand{\clspn}{\bar{\spn}\,}
\newcommand{\inv}{^{-1}}
\newcommand{\what}{\widehat}

\newcommand{\wilde}{\widetilde}
\newcommand{\xt}{\otimes}

\renewcommand{\)}{\textup{)}}
\newcommand{\cst}{\ensuremath{C^*}}
\newcommand{\csta}{\cst-algebra}
\newcommand{\cstg}{\ensuremath{\cst(G)}}
\newcommand{\cog}{\ensuremath{C_0(G)}}
\newcommand{\covrep}{covariant representation}
\newcommand{\rep}{representation}

\newcommand{\ltwo}{\ensuremath{L^2}}

\newcommand{\thmref}[1]{Theorem~\ref{#1}}
\newcommand{\secref}[1]{Section~\ref{#1}}

\newcommand{\dAK}{\epsilon}

 \usepackage[colorinlistoftodos,textsize=tiny,textwidth=1in]{todonotes}

\begin{document}
\begin{abstract}
  We present a new method of establishing a bijective correspondence
 --- in fact, a lattice isomorphism --- 
  between action- and coaction-invariant ideals of $C^*$-algebras and
  their crossed products by a fixed locally compact group.  It
  is known that such a correspondence exists whenever the
  group is amenable; our results hold for any locally compact group
  under a natural form of coaction invariance.
\end{abstract}

\title[The ladder technique]{Bijections between sets of invariant ideals,
via the ladder technique}

\author[Gillespie]{Matthew Gillespie}
\address{School of Mathematical and Statistical Sciences
\\Arizona State University
\\Tempe, Arizona 85287}
\email{mjgille1@asu.edu}

\author[Kaliszewski]{S. Kaliszewski}
\address{School of Mathematical and Statistical Sciences
\\Arizona State University
\\Tempe, Arizona 85287}
\email{kaliszewski@asu.edu}

\author[Quigg]{John Quigg}
\address{School of Mathematical and Statistical Sciences
\\Arizona State University
\\Tempe, Arizona 85287}
\email{quigg@asu.edu}

\author[Williams]{Dana P. Williams}
\address{Department of Mathematics\\ Dartmouth College \\ Hanover, NH
  03755-3551 USA}
\email{dana.williams@Dartmouth.edu}

\date{May 21, 2024}


\maketitle

\noindent

\section{Introduction}\label{sec:intro}

One of the most fundamental things to know about a \csta\ is its ideal
structure. If the \csta\ arises as the crossed product
$A\rtimes_\alpha G$ of an action $(A,G,\alpha)$, it is natural to
compare the ideal structures of $A$ and $A\rtimes_{\alpha}G$.
  This is most easily done when we restrict attention to ideals that
  are related in some way to the action of~$G$. More precisely, we
focus on ideals $I$ of $A$ that are
$\alpha$-invariant. Then the crossed product  is (isomorphic to) an
ideal $I\rtimes_\alpha G$ of $A\rtimes_\alpha G$,
and the obvious question is which ideals
of $A\rtimes_\alpha G$ arise in this way. It turns out 
that they are
precisely those ideals that are invariant under the dual coaction
$\what\alpha$, and there is a bijection between the two sets
of invariant ideals. When $G$ is amenable this is an old result of
Gootman and Lazar (\cite{gootman}), and in this paper we prove it in complete generality (Theorem~\ref{ladder thm}\,(a)). 
Dually, starting with a coaction $(A,\delta)$ of $G$, the
$\delta$-invariant ideals of $A$ correspond to $\what\delta$-invariant
ideals of the crossed product $A\rtimes_\delta G$. Again,
Gootman--Lazar proved this for amenable $G$, and in this paper we
prove it in general (Theorem~\ref{ladder thm}\,(b)).  Nilsen (\cite{nilsen}) 
has a related pair of results, but their
relationship to ours is complicated because she used a different notion of
coaction invariance than we do.

Perhaps more important, though, is our method of proof---we
introduce what we call the ``ladder technique'', which leads to quick
proofs of the aforementioned bijections, using only crossed-product
duality and the Rieffel correspondence between ideals of Morita
equivalent \csta s.  We expect the ladder technique to be useful in
other situations; for example three of us are currently working on an
application to ideals of Fell bundles over groupoids.

We begin in \secref{sec:prelim} with a detailed overview of the
crossed-product duality theorems we need, the essential definitions
and facts regarding invariant ideals, and an identification
of the associated Rieffel correspondences.  Then in
\secref{sec:ladder} we prove our main theorem, which actually
comprises four versions: starting with actions we can consider either full or
reduced crossed products, and dually we can start with
either maximal or normal coactions.  Finally, in
\secref{sec:conclusion} we close with a brief discussion comparing our
results to those of Gootman--Lazar and Nilsen.  And we point out that
one of our theorems was proved, using somewhat more technical methods
(unrelated to the ladder technique), in a recent paper two of us wrote
with Tron Omland.

\section{Preliminaries}\label{sec:prelim}

Below we will recall suitable versions of the Imai--Takai and Katayama
duality theorems for crossed products.  But to prepare for this we
start with some background on actions and coactions. 
Further details about action crossed products can be found in
  \cite{danacrossed}. Our main references for coactions and their crossed
  products are \cite[Appendix~A]{enchilada}, \cite{maximal},
  \cite{raeburn}, \cite{clda}, \cite{klqfunctor}, and \cite{nilsen}.

Throughout, $G$ is a locally compact group and $A$ is a $C^*$-algebra.
We write $\ltwo$ for $L^2(G)$, $\lambda$ and $\rho$ for the left and
right regular \rep s of $G$ on $\ltwo$, respectively, $\K$ for the
compact operators $\K(\ltwo)$, and $M$ (sometimes) for the \rep\ of
\cog\ on $\ltwo$ by multiplication operators.  We write $w_G$ for the
unitary element of $M(\cog\xt\cstg)=C_b^\beta(G,M(\cstg))$ given by
the norm-bounded strictly continuous function
$w_G(s)=s$ for $s\in G$, where the $\beta$ signifies that we are using
the strict topology on $M(\cstg)$.

We use
$(A,\alpha)$ to denote an action $\alpha$ of $G$ on $A$.
A \emph{coaction} $(A,\delta)$ of $G$ on $A$ is an injective nondegenerate
homomorphism $\delta: A \to M(A \xt \cstg)$ satisfying:
\begin{enumerate}
\item $\clspn\delta(A)(1 \xt \cstg) = A \xt \cstg$
\item
  $(\delta \xt \id) \circ \delta = (\id \xt \delta_G) \circ \delta$,
\end{enumerate}
where $\xt$ always denotes the minimal \cst-tensor product, and where
$\delta_G: \cstg \to M(\cstg \xt \cstg)$ is the usual comultiplication
of $\cstg$, i.e., the integrated form of the strictly continuous
unitary homomorphism $s\mapsto s\xt s$.  
Here, and similarly throughout the paper (unlike the
convention used in~\cite{enchilada}),
$\delta(A)(1 \xt \cstg)$ represents the set of products
 $\{ \delta(a)(1\xt z) \mid a\in A, z\in C^*(G) \}$. 
(Also note that by definition, our coactions are \emph{nondegenerate}
in the sense of \cite[Definition~2.10]{enchilada}.)
The (full) crossed
product of an action $(A,\alpha)$ is $A\rtimes_\alpha G$, which
comes with a universal \covrep\ $(i_A,i_G)$ and a dual coaction
$\what\alpha$.  The reduced crossed product of $(A,\alpha)$ is
$A\rtimes_{\alpha,r} G$.  The crossed product of a coaction
$(A,\delta)$ is $A\rtimes_\delta G$, which comes with a universal
\covrep\ $(j_A,j_G)$ and a dual action $\what\delta$.

For any action $(A,\alpha)$, the \emph{canonical surjection}
$\Phi_\alpha\colon\axg\to A\xt\K$
is determined by
\begin{enumerate}
\item
$\Phi_\alpha\circ j_{A\rtimes_\alpha G}\circ i_A=(\id\xt M)\circ \alpha\inv$

\item
$\Phi_\alpha\circ j_{A\rtimes_\alpha G}\circ i_G=1\xt\lambda$

\item
$\Phi_\alpha\circ j_G=1\xt M$.
\end{enumerate}
In item~(a), ``$\alpha^{-1}$'' refers to the map that sends
$a\in A$ to the element of $M(A\otimes C_0(G))$ determined by
the function $s\mapsto \alpha_{s^{-1}}(a)$ (see~\cite{raeburn}).
For a coaction $(A,\delta)$,
the \emph{canonical surjection}
$\Phi_\delta:\adg\to A\xt\K$ is determined by
\begin{enumerate}
\item
$\Phi_\delta\circ i_{A\rtimes_\delta G}\circ j_A=(\id\xt \lambda)\circ \delta$

\item
$\Phi_\delta\circ i_{A\rtimes_\delta G}\circ j_G=1\xt M$

\item
$\Phi_\delta\circ i_G=1\xt \rho$.
\end{enumerate}

An \emph{equivariant homomorphism} $\phi:(A,\delta)\to (B,\epsilon)$ 
between coactions
is a possibly degenerate homomorphism $\phi$ mapping $A$ into $B$ (not $M(B)$)
such that the diagram
\[
\begin{tikzcd}
A \arrow[r,"\delta"] \arrow[d,"\phi"']
&\wilde M(A\xt\cstg) \arrow[d,"\phi\xt\id"]
\\
B \arrow[r,"\epsilon"']
&\wilde M(B\xt\cstg)
\end{tikzcd}
\]
commutes (see \cite[Definition~3.2]{klqfunctor}).
Immediately following \cite[Corollary~3.13]{klqfunctor} it is 
noted that 
``a routine adaptation of the usual arguments''
(i.e., carefully applying the definitions)
shows that every $\delta-\epsilon$ equivariant homomorphism $\phi:A\to B$
gives rise to a $\what\delta-\what\epsilon$ equivariant homomorphism
\[
\phi\rtimes G=(j_B\circ\phi)\times j_G:A\rtimes_\delta G\to B\rtimes_\epsilon G.
\]

A coaction $(A,\delta)$ is
\emph{maximal} if 
$\Phi_\delta$ is an
isomorphism, and is \emph{normal} if $j_A:A\to M(A\rtimes_\delta G)$
is injective.  Equivalently, $\delta$ is normal exactly when
  $\Phi_\delta$ factors through an isomorphism of
  $A\rtimes_\delta G\rtimes_{\what\delta,r} G$ onto $A\xt\K$.

Every dual coaction $(A\rtimes_\alpha G,\what\alpha)$ is maximal.
Every coaction $(A,\delta)$ has a \emph{normalization}
$(A^n,\delta^n)$, which is unique up to isomorphism, such that
$A^n$ is a
quotient of~$A$, the coaction $\delta^n$ is normal, the quotient map
$\psi_\delta:A\to A^n$ is $\delta-\delta^n$ equivariant, 
and the  homomorphism
$\psi_\delta\rtimes G:A\rtimes_\delta G\to A^n\rtimes_{\delta^n} G$ is an
isomorphism.  For a dual coaction $\what\alpha$
on an action crossed product $A\rtimes_\alpha G$, the
associated map $\psi_{\what\alpha}$ is the regular \rep\
$A\rtimes_\alpha G\to A\rtimes_{\alpha,r} G$.  Thus, the dual coaction
on the reduced crossed product is the normalization $\what\alpha^n$, and the double-dual
action on $A\rtimes_{\alpha,r} G\rtimes_{\what\alpha^n} G$ is
$\what{\what\alpha^n}$.

\medskip
\thmref{imaitakai} below is \emph{Imai--Takai duality}.  The original
version was stated in \cite{imaitakai} for reduced crossed products,
and did not use the word `coaction'.  Version~(a) below for full crossed
products appears in \cite[Theorem~7]{raeburn}.  We use
\cite{enchilada} as a convenient reference.

\begin{thm}[{\cite[Theorem~A.67]{enchilada}}]\label{imaitakai}
  For any action $(A,\alpha)$:
  \begin{enumerate}
  \item \(full-crossed-product version\) The canonical map
  $\Phi_\alpha$ is an
    $\what{\what\alpha}-(\alpha\xt\ad\rho)$ equivariant isomorphism
    of $A\rtimes_\alpha G\rtimes_{\what\alpha} G$ onto $A\xt\K$.

  \item \(reduced-crossed-product version\) There is an
    $\what{\what\alpha^n}-(\alpha\xt\ad\rho)$ equivariant isomorphism
    $\Phi_{\alpha,r}$ of 
    $A\rtimes_{\alpha,r} G\rtimes_{\what\alpha^n} G$ onto $A\xt\K$
    such that $\Phi_\alpha = \Phi_{\alpha,r}\circ (\psi_{\what\alpha}\rtimes G)$.
  \end{enumerate}
\end{thm}

\thmref{katayama} below is \emph{Katayama duality}, which is the dual
version of \thmref{imaitakai}.  The original version was
stated in \cite[Theorem~8]{katayama} for reduced coactions, but we prefer
to work with full coactions.  

\begin{thm}[{\cite[Theorem~A.69]{enchilada}}]\label{katayama}
  For any coaction $(A,\delta)$:
  \begin{enumerate}
  \item \(maximal coaction version\)
    If $\delta$ is maximal,  the canonical surjection
    $\Phi_\delta$ is a $\what{\what{\delta}}-\dAK$ 
    equivariant isomorphism of $A\rtimes_\delta G\rtimes_{\what\delta} G$
    onto $A\xt\K$.

  \item \(normal coaction version\)
    If $\delta$ is normal, there is a
    ${\what{\what\delta}\,}^n-\dAK$ equivariant isomorphism
    $\Phi_{\delta,r}$ of 
    $A\rtimes_\delta G\rtimes_{\what\delta,r} G$ onto $A\xt\K$
    such that $\Phi_\delta = \Phi_{\delta,r}\circ \psi_{\what{\what{\delta}}}$.
  \end{enumerate}
\end{thm}

The coaction $(A\otimes\K,\dAK)$  
associated to  $(A,\delta)$ in Theorem~\ref{katayama} is defined by
\[
\dAK= \ad \bigl(1\xt (M\xt\id)(w_G^*)\bigr)\circ 
(\id\xt\Sigma)\circ (\delta\xt\id),
\]
where $\Sigma:\cstg\xt\K\to \K\xt\cstg$ is the flip isomorphism
determined on elementary tensors by $a\xt b\mapsto b\xt a$.  
The statement about equivariance in part~(a)
is actually missing from \cite{enchilada};  it follows
from the analogous result for normal coactions in
\cite[Remark~5]{nilsen} and the 
equivalence of maximal and normal coactions
(\cite[Theorem~3.3]{clda}).
  
\medskip
Given an action $(A,\alpha)$,
we say that a (closed, two-sided) ideal $I$ of $A$ is \emph{$\alpha$-invariant} 
if $\alpha$ restricts to an action $\alpha|$ on $I$;
that is, if $\alpha_g(I)\subseteq I$ for each $g\in G$.
(See \cite[Section~3.4]{danacrossed} for further discussion.)
We write $\I_\alpha(A)$ for the set of 
$\alpha$-invariant ideals of $A$.

For a coaction $(A,\delta)$, an ideal $I$ of $A$ is 
\emph{$\delta$-invariant} 
if $\delta$ restricts to a coaction $\delta|$ on $I$.
By results in Section~2 of~\cite{nilsenbundle},
this is equivalent to the condition that
\begin{equation}\label{co-invt}
\clspn\delta(I)(1\otimes C^*(G)) = I\otimes C^*(G).
\end{equation}
(See also the discussion preceding Definition~3.17
in~\cite{klqfunctor}.)
We write $\I_\delta(A)$ for the set of 
$\delta$-invariant ideals of $A$.

Somewhat surprisingly (to us), the fact that crossed products of
invariant ideals are \emph{invariant} ideals has not been clearly stated
or entirely justified elsewhere in the literature.

\begin{prop}\label{invariant}  
\begin{enumerate}
\item 
For any action~$(A,\alpha)$, for each $I\in\I_\alpha(A)$
the inclusion map
$\phi:I\hookrightarrow A$
is $\alpha|-\alpha$ equivariant, and
$\phi\rtimes G$ is an isomorphism of $I\rtimes_{\alpha|} G$ 
onto an $\what\alpha$-invariant ideal $I\rtimes_\alpha G$
of $A\rtimes_\alpha G$.
Moreover, the image $I\rtimes_{\alpha,r}G$ of $I\rtimes_\alpha G$
under the regular representation of~$A\rtimes_{\alpha}G$ is an 
$\what\alpha^n$-invariant
ideal of the reduced crossed product $A\rtimes_{\alpha,r}G$. 

\item
For any coaction $(A,\delta)$, for each $I\in\I_\delta(A)$
the inclusion map
$\phi:I\hookrightarrow A$
is $\delta|-\delta$ equivariant, and
$\phi\rtimes G$ is an isomorphism of $I\rtimes_{\delta|} G$ 
onto a $\what\delta$-invariant
ideal $I\rtimes_\delta G$
of $A\rtimes_\delta G$.
\end{enumerate}
\end{prop}

\begin{proof}
(a) Except for invariance, this  is a consequence of \cite[Proposition~3.19]{danacrossed}.
For the $\what\alpha$-invariance of $I\rtimes_\alpha G$,
we have
\[
I\rtimes_\alpha G=\phi\rtimes G(I\rtimes_{\alpha|} G)
= \clspn i_A(\phi(I))i_G(C^*(G)),
\]
so that
\begin{align*}
\what\alpha(I\rtimes_\alpha G)
&=\clspn
\what\alpha\bigl(i_A(\phi(I))\bigr)\what\alpha\bigl(i_G(\cstg)\bigr)
\\&=\clspn\bigl(i_A(\phi(I))\xt 1\bigr)(i_G\xt\id)(\delta_G(\cstg))
\end{align*}
by definition of $\what{\alpha}$.  Thus,
\begin{align*}
\clspn&\what\alpha(I\rtimes_\alpha G)(1\xt\cstg)\\
&=\clspn\bigl(i_A(\phi(I))\xt 1\bigr)(i_G\xt\id)(\delta_G(\cstg))
(1\xt\cstg)\\
&=\clspn\bigl(i_A(\phi(I))\xt 1\bigr)
(i_G\xt\id)\bigl(\delta_G(\cstg)(1\xt\cstg)\bigr)
\\&=\clspn\bigl(i_A(\phi(I))\xt 1\bigr)
(i_G\xt\id)\bigl(\cstg\xt\cstg\bigr)
\\&=\clspn\bigl(i_A(\phi(I))\xt 1\bigr)
\bigl(i_G(\cstg)\xt\cstg\bigr)
\\&=\clspn\bigl(i_A(\phi(I))i_G(\cstg)\bigr)\xt\cstg
\\&=\bigl(\phi\rtimes G(I\rtimes_{\alpha|} G)\bigr)\xt\cstg
\\&=(I\rtimes_\alpha G)\xt\cstg,
\end{align*}
which shows that $I\rtimes_\alpha G$ is $\what\alpha$-invariant.
Invariance of $I\rtimes_{\alpha,r}G$ now follows from 
$\what\alpha-\what\alpha^n$ equivariance of the regular representation.

\medskip\noindent(b) 
Some of this --- apart from invariance of $I\rtimes_\delta G$ ---
is addressed in \cite[Proposition~2.1]{nilsenbundle},
but with a different convention regarding equivariant maps. 

For us,  
the restriction $\delta|$ is a coaction on $I$ by definition of invariance,
and the inclusion $\phi$ is trivially $\delta|-\delta$ equivariant.
The verification that the induced $\what{\delta|}-\what\delta$ equivariant homomorphism 
$\phi\rtimes G:I\rtimes_{\delta|} G\to A\rtimes_\delta G$
is injective follows a standard computation with covariant representations on Hilbert space
(see \cite[proof of Proposition~2.1]{nilsenbundle}, for example).

To see that $I\rtimes_\delta G = \phi\rtimes G(I\rtimes_{\delta|} G)$ 
is an ideal of $A\rtimes_\delta G$,
first note that \cite[Lemma~3.8]{klqfunctor} implies
\[
A\rtimes_\delta G
=\clspn j_A(A)j_G(\cog)
=\clspn j_G(\cog)j_A(A),
\]
so that
\begin{align*}
(I\rtimes_\delta G)(A\rtimes_\delta G)
&=\clspn j_A(I)j_G(\cog)j_A(A)
\\&=\clspn j_A(I)j_A(A)j_G(\cog)
\\&=\clspn j_A(I)j_G(\cog)
\\&=I\rtimes_\delta G,
\end{align*}
and similarly for $(A\rtimes_\delta G)(I\rtimes_\delta G)$.

Finally, to see that $I\rtimes_\delta G$ is $\what\delta$-invariant, 
for any $s\in G$ we compute:
\begin{align*}
\what\delta_s\left(I\rtimes_{\delta} G\right)
&= \what\delta_s\left(\phi\rtimes G(I\rtimes_{\delta|} G)\right)\\
&=\phi\rtimes G\bigl(\what{\delta|}_s(I\rtimes_{\delta|} G)\bigr)\\
&=\phi\rtimes G(I\rtimes_{\delta|} G)\\
&=I\rtimes_\delta G.\qedhere
\end{align*}
\end{proof}

\medskip
For our main result, \thmref{ladder thm}, we refer to the
\emph{Rieffel correspondence} (see, for example,
\cite[Proposition~3.24]{tfb}). 
This is the lattice isomorphism between the ideals of $A$ and the ideals of
$B$ that arises from a $B-A$ \ibm~$X$ by associating each ideal $I$ of $A$ 
with the  ideal $J$~of $B$  given by
\[
 J= \clspn\lip B<X\cdot I,X>.
\]
If $A$ and $B$ are equipped with actions or coactions of~$G$ and $X$
has a suitably compatible action or coaction  
(see \cite[Definitions~2.5 and~2.10]{enchilada}), 
then the Rieffel correspondence preserves invariance of ideals, and therefore
restricts to a bijection between the 
appropriately-invariant
ideals of~$A$ and those of~$B$. 
Again, since we could not find a precise statement
of this fact in the literature, we provide one here.

\begin{lem}\label{rieffel-invt}
Let $X$ be a $B-A$ \ibm, let $I$ be an ideal of~$A$,
and let $J$ be the ideal of~$B$ associated to $I$
under the Rieffel correspondence. 
\begin{enumerate}
\item If $(A,\alpha)$ and $(B,\beta)$ are actions
of~$G$ and there exists a $\beta-\alpha$ compatible
action $\gamma$ on~$X$, then $I$ is $\alpha$-invariant
if and only if $J$ is $\beta$-invariant.

\item If $(A,\delta)$ and $(B,\epsilon)$ are coactions
of~$G$ and there exists a nondegenerate
$\epsilon-\delta$ compatible
coaction $\zeta$ on~$X$, then $I$ is $\delta$-invariant
if and only if $J$ is $\epsilon$-invariant.
\end{enumerate}
\end{lem}

\begin{proof}
Part~(a) is straightforward, so we only address part~(b).
Moreover, by symmetry it suffices to prove only one implication; 
so suppose $I$ is $\delta$-invariant. 
Then $\clspn(1\xt\cstg)\delta(I) = I\xt\cstg$
(after taking adjoints in~\eqref{co-invt}), and
nondegeneracy of $\zeta$ means that
$\clspn(1\xt\cstg)\zeta(X) = X\xt\cstg$.
So, writing $M$ for $M(B\otimes\cstg)$, we have:
\begin{align*}
\clspn(1\xt\cstg)\epsilon(J)
&=\clspn(1\xt\cstg)\epsilon\bigl({}_B\<X\cdot I,X\>\bigr)\\
&=\clspn(1\xt\cstg){}_B\<\zeta(X\cdot I),\zeta(X)\>\\
&=\clspn{}_M\<(1\xt\cstg)\zeta(X)\delta(I),(1\xt\cstg)\zeta(X)\>\\
&=\clspn{}_M\<(X\xt\cstg)\delta(I),X\xt\cstg,\>\\
&=\clspn{}_M\<(X\xt\cstg)(1\xt\cstg)\delta(I),X\xt\cstg\>\\
&=\clspn{}_M\<(X\xt\cstg)(I\xt\cstg),X\xt\cstg\>\\
&=\clspn{}_B\<X\cdot I,X\>\xt\cstg
\\&=J\xt\cstg.
\end{align*}
Thus $J$ is $\epsilon$-invariant.
\end{proof}

The isomorphisms from the duality theorems~\ref{imaitakai} and~\ref{katayama} allow 
us to make the $A\otimes\K-A$ imprimitivity bimodule $A\otimes L^2$
into an $A\rtimes G\rtimes G- A$ imprimitivity bimodule
(for each of the four types of crossed products),
and when we do this we
can identify the associated ideals explicitly:

\begin{prop}\label{associated ideal}
  \begin{enumerate}
  \item Let $(A,\alpha)$ be an action, and let $I$ be an
    $\alpha$-invariant ideal of~$A$.  Then the ideal of
    $A\rtimes_\alpha G\rtimes_{\what\alpha} G$ associated to $I$ by
    the Rieffel correspondence is $I\rtimes_\alpha G\rtimes_{\what{\alpha}} G$.  
    The ideal of $A\rtimes_{\alpha,r} G\rtimes_{\what\alpha^n} G$
    associated to $I$ by the Rieffel correspondence is
    $I\rtimes_{\alpha,r} G\rtimes_{\what\alpha^n} G$.

  \item Let $(A,\delta)$ be a coaction, and let $I$ be a
    $\delta$-invariant ideal of $A$.  If $\delta$ is maximal,
    then the ideal of
    $A\rtimes_\delta G\rtimes_{\what\delta} G$ associated to $I$ by
    the Rieffel correspondence is $I\rtimes_\delta G\rtimes_{\what\delta} G$.  
    If $\delta$ is normal,
    the ideal of $A\rtimes_\delta G\rtimes_{\what\delta,r} G$
    associated to $I$ by the Rieffel correspondence is
    $I\rtimes_\delta G\rtimes_{\what\delta,r} G$.
  \end{enumerate}
\end{prop}

\begin{proof}
(a) The key observation is that image of
$I\rtimes_\alpha G\rtimes_{\what{\alpha}} G$
in $A\otimes \K$ under 
the canonical map $\Phi_\alpha$ 
coincides with the image of 
$I\rtimes_{\alpha|}G\rtimes_{\what{\alpha|}}G$ 
under the canonical map $\Phi_{\alpha|}$,
and this latter image is precisely $I\otimes\K$
by Theorem~\ref{imaitakai}\,(a). 
Since the ideal of $A\otimes\K$ associated to $I$
by the Rieffel correspondence is also precisely $I\otimes\K$,
it follows that the ideal of 
$A\rtimes_\alpha G\rtimes_{\what\alpha} G$ associated to $I$
is
\[
\Phi_\alpha^{-1}(I\otimes\K) 
= I\rtimes_\alpha G\rtimes_{\what{\alpha}} G.
\]

The other part of (a), and both parts of (b), are quite similar.
\end{proof}

\section{The ladder technique}\label{sec:ladder}

Our main result, \thmref{ladder thm} below, rests upon the following
basic observation concerning maps between sets. We record it only
for convenient reference.
\begin{lem}\label{trivial}
  If
  \[
    \begin{tikzcd}
      &W
      \\
      Z\arrow[ur,"h"]
      \\
      &Y \arrow[ul,"g"'] \arrow[uu,"v"']
      \\
      X \arrow[ur,"f"] \arrow[uu,"u"]
    \end{tikzcd}
  \]
  is a commutative diagram of sets and maps such that $u$ and $v$ are
  bijections, then $f$, $g$, and $h$ are also bijections.
\end{lem}

Recall that for any $C^*$-algebra $A$, the set $\I(A)$ of all 
(closed, two-sided) ideals of $A$ is a lattice when ordered by inclusion,
with $I\wedge J = I\cap J$ and $I\vee J = I+J$.
For any fixed action $\alpha$ or coaction $\delta$ of $G$ on $A$, the
intersection and sum of two invariant ideals are themselves
invariant, so the sets $\I_\alpha(A)$ and $\I_\delta(A)$
of appropriately invariant ideals of~$A$
each form a sublattice of $\I(A)$.

\begin{thm}\label{ladder thm}
  \begin{enumerate}
  \item For any action $(A,\alpha)$, the assignments $I\mapsto I\rtimes_\alpha G$
    and  \hbox{$I\mapsto I\rtimes_{\alpha,r} G$} define lattice isomorphisms of
$\I_\alpha(A)$ onto $\I_{\what\alpha}(A\rtimes_\alpha G)$ and
    $\I_{\what\alpha^n}(A\rtimes_{\alpha,r} G)$, respectively.

  \item For any maximal or normal coaction $(A,\delta)$, 
  the assignment $I\mapsto I\rtimes_\delta G$
defines a lattice isomorphism of $\I_\delta(A)$ onto 
    $\I_{\what\delta}(A\rtimes_\delta G)$.
  \end{enumerate}
\end{thm}

\begin{proof}
  For the first part of~(a), consider the \emph{ladder diagram}
\begin{equation}\label{ladder ideal}
    \begin{tikzcd}
      &\I_G(A\rtimes_\alpha G\rtimes_{\what\alpha}
      G\rtimes_{\what{\what\alpha}} G)
      \\
      \I_G(A\rtimes_\alpha G\rtimes_{\what\alpha} G) \arrow[ur,"h"]
      \\
      &\I_G(A\rtimes_\alpha G) \arrow[ul,"g"'] \arrow[uu,"v"']
      \\
      \I_G(A) \arrow[ur,"f"] \arrow[uu,"u"]
    \end{tikzcd}
  \end{equation}
(For simplicity here we're writing $\I_G$ for the lattice of
 appropriately-invariant ideals in each $C^*$-algebra.)
 The diagonal maps $f$, $g$, and $h$ (the ``rungs'' of the ladder) 
 are defined by Proposition~\ref{invariant}; so for example
 $f(I) = I\rtimes_\alpha G$ for $I\in \I_\alpha(A)$,
 and $g(J) = J\times_{\what\alpha}G$ 
 for $J\in \I_{\what\alpha}(A\rtimes_\alpha G)$.  
The vertical maps $u$ and $v$ come from the Rieffel correspondence
using the imprimitivity bimodules implicit in the duality theorems~\ref{imaitakai} and~\ref{katayama};
since those bimodules have suitably compatible 
actions and coactions (respectively), $u$ and $v$
are bijections by 
Lemma~\ref{rieffel-invt}, and they
make the diagram commute by Proposition~\ref{associated ideal}.
It follows from Lemma~\ref{trivial} that
all the maps in the diagram are bijections. 

Now $u$ and $v$ are in fact lattice isomorphisms because
they are bijective restrictions of  lattice isomorphisms.
Moreover, it's routine to check that
any order-preserving bijection of lattices
\emph{whose inverse is also order-preserving} is a lattice isomorphism.
Since here it's evident by construction that $f$ and $g$ are order-preserving, 
we see that $f^{-1} = u^{-1}\circ g$ is too, 
and it follows that $f$ is
a lattice isomorphism.

The other part of (a), and both parts of (b), are quite similar.
\end{proof}

\section{Conclusion}\label{sec:conclusion}

We emphasize that our primary intent in this paper is to introduce the
``ladder technique''; the bijections in \thmref{ladder thm}, while
significant in their own right, are to some extent intended to serve
as illustrative examples of the technique.  
\thmref{ladder thm} itself is not completely new: 
for example, Gootman and Lazar
proved a version for amenable $G$.
Their results are clearly immediate corollaries of \thmref{ladder thm}:

\begin{thm}[{\cite[Theorems~3.4 \& 3.7]{gootman}}]\label{gootman}
  Assume that $G$ is amenable.
  \begin{enumerate}
  \item For any action $(A,\alpha)$ of $G$,
    an ideal $J$ of
    $A\rtimes_\alpha G$ is $\what\alpha$-invariant if and only if it
    is of the form $I\rtimes_\alpha G$ for an $\alpha$-invariant ideal
    $I$ of $A$.  Moreover, $I$ is uniquely determined.

  \item For any coaction $(A,\delta)$ of $G$,
    an ideal $J$ of
    $A\rtimes_\delta G$ is $\what\delta$-invariant if and only if it
    is of the form $I\rtimes_\delta G$ for a $\delta$-invariant ideal
    $I$ of $A$.  Moreover, $I$ is uniquely determined.
  \end{enumerate}
\end{thm}
\noindent
The full-crossed-product part of \thmref{ladder thm}\,(a) has also
appeared in \cite[Theorem~8.2]{pedersen2}, where it is proved using
different techniques (involving Landstad duality).

Nilsen has  proved a pair of related results
(for arbitrary locally compact~$G$) which are \emph{not}
corollaries of \thmref{ladder thm};
nevertheless, it seems relevant to include Nilsen's results here for
comparison.
One difference is that  Nilsen used a different notion of coaction-invariance for ideals:
if $(A,\delta)$
is a coaction, we say an ideal $I$ of $A$ is \emph{weakly invariant}
if $\delta$ passes to a coaction on the quotient $A/I$.
This is properly weaker than the notion of invariance used
in the current paper (although the two coincide for amenable~$G$).
More importantly, and Nilsen's bijection was different from ours: 
she maps an ideal $J$ in a coaction crossed-product $A\rtimes_\delta G$ 
to its \emph{restriction} $j_A\inv(J)\subseteq A$, and similarly for an action and $i_A\inv(J)$. Thus, for example, given a
    (non-normal) coaction $\delta$ of $G$ on $A$, Nilsen makes the
    (nonzero) kernel of the canonical homomorphism
    $j_A:A\to M(A\rtimes_\delta G)$ correspond to the zero ideal of
    $A\rtimes_\delta G$, which is \emph{not} the way the
    bijection in \thmref{ladder thm}\,(b) works.

\begin{thm}[{\cite[Corollaries~3.2 and 3.4]{nilsen}}]\label{nilsen}
Let $G$ be a locally compact group.
  \begin{enumerate}
  \item For any coaction $(A,\delta)$ of~$G$,
    restriction gives a bijection between the $\what\delta$-invariant
    ideals of $A\rtimes_\delta G$ and the weakly $\delta$-invariant
    ideals of $A$.

  \item For any action $(A,\alpha)$ of~$G$,
    restriction gives a bijection between the weakly
    $\what\alpha$-invariant ideals of $A\rtimes_\alpha G$ and the
    $\alpha$-invariant ideals of $A$.
  \end{enumerate}
\end{thm}



\end{document}